\documentclass[11pt]{amsart}
\usepackage{euler}

\usepackage[hidelinks, colorlinks=true, allcolors=blue]{hyperref}

\usepackage{graphicx}
\usepackage[margin=1.0in,footskip=0.45in]{geometry}
\usepackage{caption}
\usepackage{amsmath, amsthm, graphicx, amssymb,verbatim,pst-all,color, amscd, amsfonts}
\usepackage[sort,nocompress]{cite}
\usepackage{tikz-cd}
\usetikzlibrary{decorations.pathmorphing} 
\usepackage{lipsum}  
\usepackage{subcaption}
\usepackage{enumitem}
\setitemize{itemsep=1em,topsep = 1em}
\setenumerate{itemsep=1em,topsep = 1em}
{
\theoremstyle{plain}
  \newtheorem{theorem}{Theorem}[section]
  \newtheorem{proposition}[theorem]{Proposition}   
  \newtheorem{lemma}[theorem]{Lemma}

}
{
\theoremstyle{definition}

  \newtheorem{example}[theorem]{Example}

  \newtheorem{definition}[theorem]{Definition}
  \newtheorem{remark}[theorem]{Remark}
  
  \newtheorem{construction}[theorem]{Construction}
}

\newcommand{\cone}{\operatorname{Cone}}

\newcommand{\SL}[1]{\operatorname{SL}_{{#1}}}

\newcommand{\bb}[1]{\mathbb{#1}}
\newcommand{\cal}[1]{\mathcal{#1}}

\newcommand{\quotient}{ / \! \! /}

\title{Toric structure of the moduli space of points in projective space}

\begin{document}
\author{Marwan Bit}
\address{Harvey Mudd College,
301 Platt Blvd.,
Claremont, CA 91711}
\email{mbit@g.hmc.edu}

\author{Javier Gonz\'alez-Anaya}
\address{Department of Mathematics and Computer Science, 
Santa Clara University, 
Santa Clara, CA 95053}
\email{jgonzalezanaya@scu.edu}

\author{Dagan Karp}
\address{Department of Mathematics, 
Harvey Mudd College,
301 Platt Blvd.,
Claremont, CA 91711}
\email{dkarp@hmc.edu}

\author{Yuanyuan Luo}
\address{Harvey Mudd College,
301 Platt Blvd.,
Claremont, CA 91711}
\email{jluo@g.hmc.edu}

\date{\today}

\begin{abstract}
Gallardo and Routis constructed compactifications of the moduli space of $n$ labeled points in $\mathbb{P}^d$ by assigning weights to points, generalizing Hassett's weighted compactifications of $M_{0,n}$ to higher-dimensional projective spaces. Among their compactifications, there is a toric compactification that generalizes the standard Losev--Manin compactification to this higher-dimensional setting. Our main result identifies the fan of this toric compactification as a \emph{symmetric product} of a nested fan, generalizing the classical connection between Losev--Manin spaces and the permutohedron to arbitrary dimension. More generally, we prove that the fans of all Gallardo–Routis compactifications that admit reduction maps from this Losev–Manin space are symmetric products of building sets. This shows that the combinatorics of these compactifications are controlled by coarsenings of the permutohedral fan that give rise to Hassett spaces.
\end{abstract}

\maketitle

\section{Introduction}
Roughly 25 years ago, Losev and Manin introduced a compactification of the moduli space of points on the projective line \cite{losev2000new}. Their space, now denoted $\overline{M}_{0,n}^{LM}$, parametrizes weighted pointed stable rational curves in which two marked points are distinguished and forbidden from colliding, while the remaining $n$ points may coincide. A central feature of their construction is that $\overline{M}_{0,n}^{LM}$ is isomorphic to the toric variety of the permutohedron. This toric structure has proved remarkably useful: Bergstr\"om and Minabe exploited permutohedral symmetry to compute the cohomology of Losev--Manin space \cite{bergstrom_minabe_2014}, and Gholampour et al.\ demonstrated the computational power of its Cremona symmetry in Gromov--Witten theory \cite{cremona2016}.

Losev--Manin spaces fit naturally into the framework of Hassett's weighted pointed stable curves $\overline{M}_{0,\mathcal{A}}$ \cite{hassett2003moduli}, corresponding to heavy/light weight vectors $\mathcal{A} = (1, 1, \epsilon, \ldots, \epsilon)$ \cite{CAVALIERI_HAMPE_MARKWIG_RANGANATHAN_2016}. This perspective has inspired considerable work on the geometry and intersection theory of Hassett spaces; for instance, Kannan et al.\ compute the Chow ring of heavy/light Hassett spaces using tropical methods \cite{kannan2021chow}. Hassett's framework has also served as a template for other moduli problems: Clader et al.\ construct permutohedral moduli spaces for curves with cyclic group actions \cite{clader2023permutohedral, clader2023wonderful}, whose dual boundary complex was further studied by Anderson \cite{anderson2024}, and Chen et al.\ study configurations of points in affine space modulo translation and scaling \cite{Chen-Gibney-Krashen}.

Gallardo and Routis extended Hassett's ideas to higher-dimensional ambient spaces in two related but distinct ways \cite{gallardo2017wonderful}. The first parametrizes points in affine space up to translation and scaling, while the second parametrizes points in projective space up to the action of $\mathrm{SL}_d$. Both constructions admit weight systems analogous to Hassett's, and both recover Hassett spaces when $d = 1$. For each construction, a particular choice of weights yields a toric compactification: in the affine case, these are the higher-dimensional Losev--Manin spaces studied in \cite{gallardo2023higherdimensionallosevmaninspacesgeometry}. In the projective case, one obtains toric varieties $\overline{P}_{d,n}^{LM}$, whose structure has yet to be further studied, and which is the focus of this paper.

Generalized permutohedra are polytopes obtained from usual permutohedra by parallel translations of its faces. Equivalently, they are polytopes whose inner normal fan is a coarsening of the permutohedral fan. They form a rich combinatorial family containing examples such as graph associahedra, nestohedra, and stellahedra. In the context of moduli theory, it is natural to ask what generalized permutohedra give rise to compactifications of $M_{0,n}$ or other moduli spaces.  This question has been recently answered by Bown and Gonz\'alez Anaya \cite{bown2024hypergraphassociahedracompactificationsmoduli}, who classified which coarsenings of the permutohedral fan give rise to Hassett compactifications $\overline{M}_{0,\mathcal{A}}$, and proved an analogous result for the higher-dimensional Losev–-Manin spaces of \cite{gallardo2023higherdimensionallosevmaninspacesgeometry}. In this paper, we solve the analogous classification problem for the moduli space of points in $\mathbb{P}^d$. Concretely, we give a complete combinatorial characterization of those coarsenings of the fan of $\overline{P}_{d,n}^{LM}$ that also give rise to toric compactifications of this moduli problem. To state our results precisely, we require some combinatorial background.

A \emph{building set} $\mathcal{B}$ on $[n] = \{1, \ldots, n\}$ is a collection of subsets closed under taking unions of overlapping sets. Building sets encode nestohedra, a class of polytopes that are notably well-suited to study the combinatorics of moduli spaces of point configurations. Given a building set $\mathcal{B}$, there is an associated nested fan $\Sigma(\mathcal{B})$ and corresponding toric variety $X(\mathcal{B})$. We introduce a product operation: for a building set $\mathcal{B}$ on the $(n-1)$-simplex, the \emph{$d$\textsuperscript{th} symmetric product} $\operatorname{sym}^d(\mathcal{B})$ is a building set on the $d$-fold product of simplices, constructed by taking products $I \times \cdots \times I$ for each $I \in \mathcal{B}$. See Section~\ref{sec: symmetric products} for precise definitions.

Our first main result identifies the combinatorial structure of the Losev--Manin spaces $\overline{P}_{d,n}^{LM}$.

\begin{theorem}[{Theorem~\ref{main thm1}}]\label{intro: main thm}
    The Losev--Manin compactification $\overline{P}_{d,n}^{LM}$ is a toric variety whose fan is the $d$\textsuperscript{th} symmetric product of the complete building set on $\{d+2,\ldots,n\}$.
\end{theorem}

For $d = 1$, this recovers Losev and Manin's original result that $\overline{M}_{0,n}^{LM}$ is the toric variety of the permutohedron. More generally, the GR construction allows points to be assigned weights $\mathcal{A} = (a_1, \ldots, a_n)$, yielding a family of compactifications $\overline{P}_{d,n}^\mathcal{A}$. Our second main result describes precisely the fans of all toric compactifications that admit a reduction map from $\overline{P}_{d,n}^{LM}$.

\begin{theorem}[{Theorem~\ref{main thm2}}]\label{intro: main thm2}
Let $\mathcal{A}$ be a weight vector such that $\overline{P}_{d,n}^\mathcal{A}$ admits a reduction map from $\overline{P}_{d,n}^{LM}$. Then, there exists a building set $\mathcal{B}_\mathcal{A}$, depending on $\mathcal{A}$, such that the fan of $\overline{P}_{d,n}^\mathcal{A}$ is the $d$\textsuperscript{th} symmetric product of $\mathcal{B}_\mathcal{A}$.
\end{theorem}

The previous theorem implies that the combinatorial data defining the compactification $\overline{P}_{d,n}^\mathcal{A}$ is fully controlled by the combinatorics of the building set $\mathcal{B}_\mathcal{A}$ together with the symmetric product construction. In fact, the toric variety associated to this building set is a toric Hassett compactification of $M_{0,n}$.

\begin{theorem}[{Theorem~\ref{main thm3}}]\label{intro: main thm3}
    Let $\mathcal{A}=(a_1,\dots,a_n)$ be a weight vector such that $\overline{P}_{d,n}^\mathcal{A}$ admits a reduction map from $\overline{P}_{d,n}^{LM}$, and let $\mathcal{B}_\mathcal{A}$ be the building set from Theorem~\ref{intro: main thm2}. Then $\mathcal{A}' = (1,a_{d+1},\dots,a_n)$ defines a Hassett compactification $\overline{M}_{0,\mathcal{A}'}$ satisfying:
    \begin{enumerate}
        \item $\overline{M}_{0,\mathcal{A}'}$ admits a reduction map to the compactification of $M_{0,n-d+1}$ given by $\mathbb{P}^{n-d-2}$, and a reduction map from $\overline{M}_{0,n-d+1}^{LM}$.
        \item $\overline{M}_{0,\mathcal{A}'}$ is the toric variety constructed from the nested fan of $\mathcal{B}_\mathcal{A}$.
    \end{enumerate}
\end{theorem}

In this way, the symmetric product provides a bridge relating the modular coarsenings of the fan of $\overline{P}_{d,n}^{LM}$ to those of $\overline{M}_{0,n-d+1}^{LM}$ that refine the fan of $\mathbb{P}^{n-d-2}$.

\subsection*{Acknowledgements.} JGA is grateful to Patricio Gallardo for many helpful conversations.

\section{The Gallardo--Routis compactification of the moduli space of labeled points in projective space}

In \cite{gallardo2017wonderful}, Gallardo and Routis studied the problem of compactifying the moduli space of configurations of labeled points in general position in $\mathbb{P}^d$ up to the action of $\SL{d+1}$. Their theory considers different compactifications by assigning weights to the points. Just as Hassett's compactifications for $M_{g,n}$, different weights may give rise to different compactifications, some of which are toric. We begin this section by outlining the construction of these moduli spaces.

Let $n$ and $d$ be positive integers such that $n>d+2$. Define the weight vector $\cal{W}=(w_1,\dots,w_n)\in\mathbb{Q}^n$ given by
\begin{align*}
w_i=\left\{
\begin{aligned}
    &1-\varepsilon',&\text{ if } 1\leq i\leq d;\\  
    &1 - (n-(d+1))\varepsilon + d\varepsilon',&\text{ if } i=d+1;\\
    &\varepsilon,&\text{ if } d+1<i;
\end{aligned}
\right.
\end{align*}
where $\varepsilon=1/(n-d)$ and $\varepsilon'=1/((d+1)(n-d))$.

Consider the $\mathbb{Q}$-Cartier divisor $\cal{O}(w_1,\dots,w_n)$ in $(\mathbb{P}^d)^n$, and let $L$ be any multiple of it that is Cartier. The line bundle $L$ is ample, and it admits a canonical $\SL{d+1}$-linearization by \cite[Chapter~3]{dolgachev2003lectures}. Gallardo and Routis in \cite[Lemma~4.2]{gallardo2017wonderful} prove that the semistable and stable loci in $(\mathbb{P}^d)^n$ agree under this linearization. Moreover, they characterize this open set as the set of points $(p_1,\dots,p_n)\in(\mathbb{P}^d)^n$ such that:
\begin{enumerate}
    \item[(C1)] The points $p_1,\dots,p_d,p_i$ are in general linear position for all $i=d+1,\dots,n$. That is, not one of these points lies in the span of the other $d$ points.
    \item[(C2)] The points $p_{d+2},\dots,p_{n}$ cannot all lie on the hyperplane spanned by $p_1,\dots,p_{k-1},p_{k+1},\dots,p_{d+1}$ for any $k=2,\dots,d$. 
\end{enumerate}
Let $\cal{U}\subseteq(\mathbb{P}^d)^n$ denote the open subset defined by these conditions. Condition (C1) implies that the $\SL{d+1}$-equivalence class of a point $(p_1,\dots,p_n)\in\cal{U}$ has a representative of the form
\[
    p_1=[1:0:\cdots:0],\dots,p_{d+1}=[0:\cdots:0:1],\quad\text{and}\quad p_i=[b_i^0:\cdots:b_i^{d-1}:1],
\]
for all $i=d+2,\dots,n$. By condition (C2), in the previous representation the $(n-d-2)$-tuple $(b_{d+2}^k,\dots,b_n^k)\in\mathbb{C}^{n-d-2}$ is not identically zero for any $k=0,\dots,d-1$.

From these observations it follows that there is a well-defined morphism
\[
    \cal{U}/\SL{d+1}\to(\mathbb{P}^{n-d-2})^d
\]
mapping the $\SL{d+1}$-equivalence class of $(p_1,\dots,p_n)\in\cal{U}$ in $(\mathbb{P}^d)^n\quotient_L\SL{d+1}$ to the point
\[
    \prod_{k=0}^{d-1}[b^k_{d+2}:\cdots:b_n^k]\in(\mathbb{P}^{n-d-2})^d.
\]
Moreover, by \cite[Lemma~4.2]{gallardo2017wonderful}, this map extends to an isomorphism
\[
    (\mathbb{P}^d)^n\quotient_L\SL{d+1}\cong(\mathbb{P}^{n-d-2})^d.
\]

\begin{example}
    Let us consider the case of $n=5$ and $d=2$. Then, the isomorphism
    \[
    (\mathbb{P}^2)^5\quotient_L\SL{3}\cong(\mathbb{P}^{1})^2
    \]
    maps a point $([b_4^0:b_5^0],[b_4^1:b_5^1])\in(\mathbb{P}^{1})^2$ to the equivalence class of the coordinate points $p_1,p_2,p_3$, and
    \[
        p_4=[b_4^0:b_4^1:1],\, p_5=[b_5^0:b_5^1:1].
    \]
\end{example}
The isomorphism $(\mathbb{P}^d)^n\quotient_L\SL{d+1}\cong(\mathbb{P}^{n-d-2})^d$ realizes $(\mathbb{P}^{n-d-2})^d$ as a parameter space of $n$ labeled points in $\mathbb{P}^d$ satisfying the conditions (C1) and (C2), up to the action of $\SL{d+1}$. In the previous construction the weight vector $\cal{W}$ determines which point configurations of points are stable. Next we will explain how to construct other compactifications of this moduli problem by using more general weight vectors. 

The variety $(\mathbb{P}^d)^n\quotient_L\SL{d+1}\cong(\mathbb{P}^{n-d-2})^d$ is denoted by $\overline{P}_{d,n}^\mathcal{W}$, and its points parametrize equivalence classes of point configurations satisfying conditions (C1) and (C2). Having established this first compactification, we now consider more general weight vectors that allow certain points to coincide, leading to different moduli spaces. 

\begin{definition}
    We refer to a vector $\cal{A}=(a_1,\dots,a_n)\in\mathbb{Q}^n$ such that $w_i\leq a_i\leq 1$ for all $i=1,\dots, n$ as a \emph{weight vector}. The set of all weight vectors is denoted $\cal{D}_{d,n}^P$.
\end{definition}

Given a weight vector $\cal{A}=(a_1,\dots,a_n)$, a configuration of points $(p_1,\dots,p_n)$ is said to be $\cal{A}$-stable if $p_{i_1}=\cdots=p_{i_k}$ implies that $a_{i_1}+\cdots+a_{i_k}\leq 1$. For any $\cal{A}\in\cal{D}_{d,n}^P$, Gallardo and Routis construct a smooth moduli space of $\cal{A}$-stable point configurations via GIT, denoted $\overline{P}_{d,n}^\cal{A}$. We do not need the precise construction of this object. Instead, we recall an alternative construction central to our main result.

\begin{proposition}[{\cite[Corollary~4.9]{gallardo2017wonderful}}]
    Let $\cal{A}\in\cal{D}_{d,n}^P$. Then, $\overline{P}_{d,n}^\cal{A}$ is isomorphic to a sequence of blow-ups of $(\mathbb{P}^{n-d-2})^d$ along the iterated strict transforms of a collection $\cal{G}_\cal{A}$ of multilinear subspaces; see Definition~\ref{def: ga}. These iterated blow-ups can be performed in any order of ascending dimension of the elements of $\cal{G}_\cal{A}$.
\end{proposition}

\begin{remark}
     The boundary strata of $\overline{P}_{d,n}^\cal{A}$ are the union of smooth irreducible divisors, and any set of these boundary divisors intersects transversally.
\end{remark}

To construct the set $\cal{G}_\cal{A}$ we identify which subvarieties of $(\mathbb{P}^{n-d-2})^d$ correspond to configurations where specific points coincide. The locus in $(\mathbb{P}^{n-d-2})^d$ corresponding to the $\SL{d+1}$-equivalence classes of point configurations where a point $p_i$ with $i=d+2,\dots,n$ coincides with $p_{d+1}$ is
\[
    H_{i,d+1} = V(b_i^0,\dots,b_i^{d-1}),
\]
while the locus where two distinct points $p_i$ and $p_j$ with $i,j=d+2,\dots,n$ coincide is  
\[
    H_{i,j} = V(b_i^0-b_j^0,\dots,b_i^{d-1}-b_j^{d-1}).
\]
In general, given a subset $I\subsetneq\{d+1,\dots,n\}$ such that $|I|\geq 2$, the locus where the points $p_i$ with $i\in I$ coincide is the subvariety
\begin{align*}
    H_I = \left\{
    \begin{aligned}
        &\bigcap_{i\in I\setminus d+1} H_{i,d+1}, &\text{if }d+1\in I;\\
        &\;\;\,\bigcap_{i,j\in I} H_{i,j}, &\text{if }d+1\notin I.
    \end{aligned}
    \right.
\end{align*}

\begin{definition}\label{def: ga}
    Consider a weight vector $\cal{A}\in\cal{D}_{d,n}^P$. Define the set
    \[
        \cal{G}_\cal{A} = \left\{
            H_I\subseteq (\mathbb{P}^{n-d-2})^d\,\vert\, I\subsetneq\{d+1,\dots,n\}\text{ and } \sum_{i\in I}a_i>1
        \right\}.
    \]
\end{definition}

\begin{example}\label{ex: moduli}
    Consider $n=5$ and $d=2$. 
    \begin{enumerate}
        \item Let $\cal{A}=(1,\dots,1)\in\cal{D}_{2,5}^P$. Then, $\overline{P}_{2,5}=\overline{P}_{2,5}^\cal{A}$ is the iterated blow-up of $(\mathbb{P}^1)^2$ along the subvarieties
        \[
            V(x_0,y_0),V(x_1,y_1),V(x_0-y_0,x_1-y_1),
        \]
        where $[x_0:x_1]$ and $[y_0:y_1]$ are the coordinates of each $\mathbb{P}^1$.

        \item Let $\cal{B}=(1,1,1,1/2,1/2)\in\cal{D}_{2,5}^P$. Then, $\overline{P}_{2,5}^\cal{B}$ is the iterated blow-up of $(\mathbb{P}^1)^2$ along the subvarieties
        \[
            V(x_0,y_0),V(x_1,y_1).
        \]
    \end{enumerate}
\end{example}

Among all possible weight vectors, there is a particular choice of weight data giving rise to a higher-dimensional analog of the Losev–Manin compactification of $M_{0,n}$.

\begin{definition}
    Consider the weight vector $\cal{A}_{LM}=(a_1,\dots,a_n)$ such that
    \begin{align*}
        a_i = \left\{
        \begin{aligned}
            &1,&\text{ if }i=1,\dots,d+1;\\
            &1/(n-d-1),&\text{if } i=d+2,\dots,n.
        \end{aligned}
        \right.
    \end{align*}
    We refer to the compactification $\overline{P}_{d,n}^{LM}:=P_{d,n}^{\cal{A}_{LM}}$ as the \emph{Losev--Manin compactification} for this moduli problem.
\end{definition}    

\begin{lemma}
    The Losev--Manin compactification $\overline{P}_{d,n}^{LM}$ is a toric variety.
\end{lemma}

\begin{proof}
    By definition, the Losev--Manin compactification is the iterated blow-up of $(\mathbb{P}^{n-d-2})^d$ along the strict transforms of the subvarieties
    \[
        \bigcap_{i\in I\setminus d+1} H_{i,d+1}=
        \bigcap_{i\in I\setminus d+1} V(b_i^0,\dots,b_i^{d-1})
    \]
     Each one of these subvarieties is torus-invariant in $(\mathbb{P}^{n-d-2})^d$ under its standard toric structure. The result follows.
\end{proof}

The reason we call these compactifications Losev--Manin is that for $d=1$ the resulting space is precisely the Losev--Manin compactification of $M_{0,n}$.

\begin{proposition}
    The variety $\overline{P}_{1,n}^{LM}$ is isomorphic to the standard Losev--Manin compactification $\overline{M}_{0,n}^{LM}$.
\end{proposition}
 \begin{proof}
     The variety $\overline{P}_{1,n}^{LM}$ is isomorphic to the iterated blow-up of $\mathbb{P}^{n-3}$ along the strict transforms of the subvarieties
        \[
            \bigcap_{i\in I\setminus 2} H_{i,d+1}=
            \bigcap_{i\in I\setminus 2} V(b_i^0),
        \]
    as $I$ ranges among all subsets $\{3,\dots,n\}$ of size larger than two, in any order of non-decreasing cardinality. Here the coordinates of $\mathbb{P}^{n-3}$ are $[b_{3}^0:\cdots:b_n^0]$. This process consists on blowing up the invariant points, then the strict transform of all invariant lines, etc. The resulting variety is precisely the Losev--Manin space.
\end{proof}

\section{Symmetric products of nestohedra}\label{sec: symmetric products}

\subsection*{Notation} Define $[n]:=\{1,\dots,n\}$. We denote the $n$-dimensional simplex by $\Delta_n$, and realize its inner normal fan in the vector space $\mathbb{R}^{n+1}/\mathbb{R}(1,\dots,1)$ as the fan with rays generated by the standard basis vectors $e_1,\dots,e_{n+1}$.

In this section we introduce the notion of symmetric product for nestohedra. This notion captures precisely the combinatorial structure of all toric compactifications $\overline{P}_{d,n}^{LM}$ and more. We first recall some basic facts about nestohedra and J. Almeter's notion of $P$-nestohedra \cite{almeter-thesis,almeter2020generalizing}.

\subsection{Nestohedra}

Recall that truncating a face $F$ of a polytope $P$ corresponds to performing a barycentric subdivision of the normal fan $\Sigma$ of $P$ along the cone $\sigma_F\in\Sigma$, which in turn corresponds to blowing up the associated torus-invariant subvariety in the toric variety $X(\Sigma)$.

Nestohedra are polytopes obtained from the simplex by truncating only certain faces, determined by a building set.

\begin{definition}\label{bg: definition building set}
    A collection $\cal{B}\subseteq2^{[n]}$ is said to be a (combinatorial) \emph{building set} if the following two conditions are met:
    \begin{itemize}
        \item $\{i\}\in\cal{B}$ for every singleton $i\in [n];$
        \item if $I,J\in\cal{B}$ satisfy $I\cap J\neq\emptyset$, then $I\cup J\in\cal{B}$.
    \end{itemize}
\end{definition}

The elements of the collection $\mathcal{B}_{max}\subseteq\cal{B}$ of inclusion-maximal elements of $\cal{B}$ are called the connected components of $\cal{B}$. A building set is said to be \emph{connected} if it has a single connected component, $\cal{B}_{max}=\{[n]\}$.

A building set $\cal{B}$ over $[n]$ gives rise to a $(n-1)$-dimensional polytope, called \emph{nestohedron}. In the connected case, these polytopes can be constructed via a series of truncations of the $(n-1)$-dimensional simplex as follows; see \cite[Theorem~4]{feichtner-yuzvinsky}.

\begin{construction}\label{bg: fan of connected building set}
    Start with the fan of the $(n-1)$-dimensional simplex, $\Sigma_{n-1}\subset\bb{R}^{n}/\bb{R}(1,\dots,1)$, whose rays have been labeled by the singletons $i$ for $i\in [n]$. Given a nonsingleton $I\in\cal{B}$, define the cone $\sigma_I:=\cone(e_i\,\vert\, i\in I)\in\Sigma_{n-1}$. Finally, define a total order of $\cal{B}=\{I_1,\dots,I_k\}$ such that $|I_i|\geq |I_j|$ if $i<j$. Then, by \cite[Theorem~4]{feichtner-muller}, the fan obtained after performing iterated barycentric subdivisions of the cones $\sigma_I$ as $I$ ranges through the elements of $\cal{B}$ in our choice of total order is called the \emph{nested fan of $\mathcal{B}$}, and we denote it as $\Sigma(\cal{B})$.
\end{construction}

The previous construction of $\Sigma(\cal{B})$ is independent of our choice of total order, provided that larger sets come before smaller ones. The nested fan is unimodular, as the barycentric subdivision of unimodular cones is unimodular; see also \cite[Corollary~5.2]{zelevinsky-nested}.

\begin{definition}
    The fan $\Sigma(\cal{B})$ is the inner normal fan of a polytope by construction. We call this polytope the \emph{nestohedron} corresponding to $\cal{B}$. We denote the toric variety corresponding to the nested fan as $X(\mathcal{B})$.
\end{definition}

\subsection{$P$-nestohedra}

Jordan Almeter introduced a very natural generalization of nestohedra in their doctoral dissertation \cite{almeter-thesis}, see also \cite{almeter2020generalizing} for a streamlined account. Their construction works as follows. Let $P$ be a simple, full-dimensional polytope in $\bb{R}^{n}$. Further suppose the facets of $P$ have been labeled by a set $S$.

\begin{definition}
    A \emph{$P$-building set} is a collection $\cal{B}\subseteq 2^S$ such that
    \begin{itemize}
        \item $\{i\}\in\cal{B}$ for every singleton $i\in S;$
        \item for each $I\in\cal{B}$, the face $F_I$ of $P$ is nonempty;
        \item for any two sets $I,J\in\cal{B}$ such that $I\cap J\neq\emptyset$, if $F_I\cap F_J=F_{I\cup J}$ is not empty, then $I\cup J\in\cal{B}$.
    \end{itemize}
\end{definition}

In analogy with the case of standard building sets, a $P$-building set defines a polytope which is constructed as an iterated truncation of $P$, called the \emph{$P$-nestohedron} of $\cal{B}$. In particular, if $P$ is a simplex, then the previous definition specializes to Definition~\ref{bg: definition building set}. The inner normal fan of the $P$-nestohedron of $\cal{B}$ is obtained as follows.

\begin{construction}
    Start with the inner normal fan of $P$, denoted $\Sigma_{P}\subset\bb{R}^{n}$, whose rays have been labeled by the singletons $i$ for $i\in S$. Given a nonsingleton element $I\in\cal{B}$, define the cone $\sigma_I:=\cone(e_i\,\vert\, i\in I)\in\Sigma_{P}$. Finally, define a total order of $\cal{B}=\{I_1,\dots,I_k\}$ such that $|I_i|\geq |I_j|$ if $i<j$. The resulting fan is called the \emph{$P$-nested fan of $\mathcal{B}$}, and we denote it as $\Sigma(\cal{B})$.
\end{construction}

By \cite[Theorem~3.4]{feichtner2004incidence}, the $P$-nested fan is unimodular, and independent of our choice of total order.

\begin{definition}
    Given a $P$-building set $\cal{B}$, the fan $\Sigma(\cal{B})$ is the inner normal fan of a polytope by construction. We call this polytope the \emph{$P$-nestohedron} corresponding to $\cal{B}$. We denote the toric variety corresponding to the nested fan as $X(\mathcal{B})$.
\end{definition}

\subsection{Symmetric products of nestohedra}

We now present the main combinatorial construction of this paper:

\begin{definition}[Symmetric products]\label{def:symmetric_product}
Fix positive integers $d$ and $n$. Let $\mathcal{B}$ be a connected building set on the $(n-1)$-dimensional simplex $\Delta$, and define $P = \Delta^d$ to be the $d$\textsuperscript{th} Cartesian product of $\Delta$. Label the facets of $\Delta$ with $[n]$, and those of $P$ with the disjoint union $[n]^{\sqcup d} := \{(i,j) \mid i \in [n], j \in [d]\}$. 

The \emph{$d$\textsuperscript{th} symmetric product of $\mathcal{B}$} is the $P$-nestohedron constructed from the $P$-building set consisting of all singletons of $[n]^{\sqcup d}$, together with all subsets of the form
\[
I^{\sqcup d} \subseteq [n]^{\sqcup d}, \quad \text{for all } I \in \mathcal{B}.
\]
We denote this $P$-building set as $\operatorname{sym}^d(\mathcal{B})$.
\end{definition}

\begin{remark}
    The term \emph{symmetric} reflects the fact that the construction treats all $d$ factors of $\Delta^d$ identically. This captures the essence of the GR compactifications, which themselves exhibit symmetry across the $d$ factors of $(\mathbb{P}^{n-d-2})^d$.
\end{remark}

\begin{example}\label{example sym prod}
    Let $n=1$ and $d=2$, and label the facets of $\Delta_1$ by $\{1,2\}$, so that its inner normal fan consists of two rays generated by vectors $e_1$ and $e_2$ such that $e_1+e_2=0$. Consider the building set $\mathcal{B}=2^{[2]}\setminus\emptyset$ on $\Delta_1$. Then, the corresponding symmetric product is the iterated barycentric subdivision of the fan $\Sigma_1\times\Sigma_1$ along the cones
    \[
        (\mathbb{R}e_1,\mathbb{R}e_1)\quad\text{and}\quad(\mathbb{R}e_2,\mathbb{R}e_2).
    \]
    Notice that the corresponding toric variety is precisely the Losev--Manin compactification $\overline{P}_{2,5}^{LM}$ obtained with the second weight vector in Example~\ref{ex: moduli}. See Figure~\ref{fig:example}.
\end{example}   

\begin{figure}[htbp]
    \centering
    \includegraphics[width=0.7\textwidth]{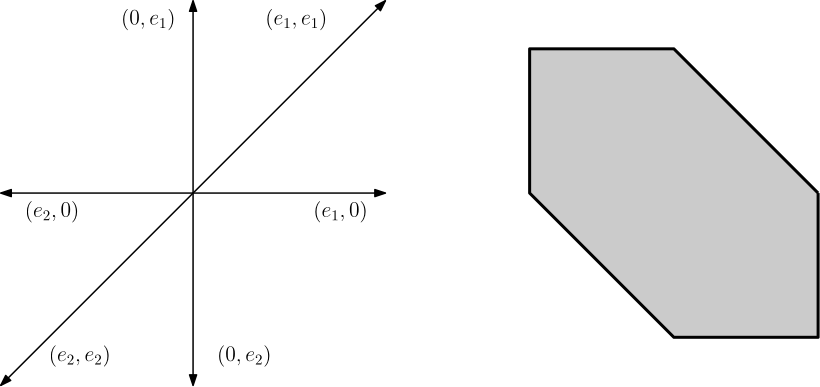}
    \caption{The fan $\operatorname{sym}^2(\mathcal{B})$ and corresponding polytope in Example~\ref{example sym prod}.}
    \label{fig:example}
\end{figure}

\begin{remark}
    Note that the product of the singleton sets from the building set $\mathcal{B}$ gives rise to nontrivial cones in $\operatorname{sym}^d(\mathcal{B})$, as shown in the previous example.
\end{remark}

Equipped with this definition, we are now ready to prove Theorem~\ref{intro: main thm}.

\begin{theorem}[{Theorem~\ref{intro: main thm}}]\label{main thm1}
    The Losev--Manin compactification $\overline{P}_{d,n}^{LM}$ is a toric variety whose fan is the $d$\textsuperscript{th} symmetric product of the complete building set on $\{d+2,\ldots,n\}$.
\end{theorem}

\begin{proof}
    Fix positive integers $n$ and $d$ such that $n>d+2$. Then, the Losev--Manin compactification $\overline{P}_{d,n}^{LM}$ is the iterated blow-up of $(\mathbb{P}^{n-d-2})^d$ along the strict transforms of the subvarieties
    \[
        H_{I\cup\{d+1\}}=\bigcap_{i\in I} V(b_i^0,\dots,b_i^{d-1}),
    \]
     for all $I\subsetneq\{d+2,\dots,n\}$ such that $|I|\geq 1$. Here we remind the reader that we are indexing the elements of this product space by $k=0,\dots,d-1$, and the homogeneous coordinates of the $k$\textsuperscript{th} copy of $\mathbb{P}^{n-d-2}$ are $[b_{d+2}^k:\ldots:b_{n}^k]$. It follows that each one of these subvarieties is torus-invariant in $(\mathbb{P}^{n-d-2})^d$ under its standard toric structure.

     On the other hand, consider $(n-d-2)$-dimensional simplex $\Delta$ with facets labeled by $S=\{d+2,\ldots,n\}$. The fan of the $d$\textsuperscript{th} symmetric product of $\mathcal{B}=2^{S}\setminus\emptyset$ consists of the cones
     \[
        \tau_I:=(\sigma_I,\sigma_I,\ldots,\sigma_I)\subseteq\left(\Sigma_{n-d-2}\right)^d
     \]
    for all $I\subsetneq S$, where $\sigma_I\in\Sigma_{n-d-2}$ is the cone defined by $I$.

     The key observation is that $H_I$ corresponds to $\tau_I$ under the toric orbit-cone correspondence. Indeed, for each $i\in S$ the subvariety $V(b_i^0,\dots,b_i^{d-1})$ corresponds to the cone obtained from the cone $(\rho_i,\dots,\rho_i)\in\left(\Sigma_{n-d-2}\right)^d$, where $\rho_i\in\Sigma_{n-d-2}$ is the ray corresponding to $i\in S$. On the other hand, by Lemma~\ref{lemma toric intersection} below, the intersection $H_I$ corresponds to the cone $(\eta_I,\dots,\eta_I)$, where
     \[
        \eta_I = \bigvee_{i\in I} \rho_i,
     \]
     is the smallest cone in $\Sigma_{n-d-2}$ containing all rays $\rho_i$ for $i\in I$. Since $\Sigma_{n-d-2}$ is the fan of the simplex, $\eta_I$ is precisely the cone $\sigma_I$ from above. Therefore, $H_I$ is the invariant subvariety defined by $\tau_I$, as claimed. This proves the result.
\end{proof}

\begin{lemma}\label{lemma toric intersection}
   Consider a toric variety $X=X_\Sigma$ with fan $\Sigma$, and consider any two cones $\sigma,\tau\in\Sigma$. Denote by $V(\sigma) = \overline{O(\sigma)}$ the invariant subvariety corresponding to $\sigma$, and the same for $\tau$. If it exists, let $\sigma\vee\tau$ be the minimal cone in $\Sigma$ containing both $\sigma$ and $\tau$. Then,
    \[
        V(\sigma)\cap V(\tau) = V(\sigma\vee\tau),
    \]
    provided that $\sigma\vee\tau$ exists, and the intersection is empty otherwise.
\end{lemma}
\begin{proof}
    By \cite[Theorem~3.2.6]{CLS}, $V(\sigma) = \bigcup_{\sigma\preceq\eta}O(\eta)$. Then,
    \[
        V(\sigma)\cap V(\tau) = \bigcup_{\substack{\sigma\preceq\eta\\ \tau\preceq\eta}}O(\eta).
    \]
    This equals $\bigcup_{\sigma\vee\tau \preceq\eta}O(\eta) = V(\sigma\vee\tau)$ if $\sigma\vee\tau$ exists, and is empty otherwise. The result follows.
\end{proof}

Upon further inspection of the proof of Theorem~\ref{intro: main thm} above, it becomes clear that the same argument hinges on the symmetry of the subvarieties $H_I$ in $\left(\mathbb{P}^{n-d-2}\right)^d$, rather than on the specific fan of the Losev--Manin compactification. In particular, an identical argument to the one used above yields Theorem~\ref{intro: main thm2}.

\begin{theorem}[{Theorem~\ref{intro: main thm2}}]\label{main thm2}
    Let $\mathcal{A}=(a_1,\dots,a_n)\in\mathcal{D}_{d,n}^P$ be a weight vector such that $a_{d+2}+\cdots+a_n\leq 1$, so that $\mathcal{G}_\mathcal{A}\subseteq\mathcal{G}_{\mathcal{A}_{LM}}$. Define the building set $\mathcal{B}_\mathcal{A}$ on $S=\{d+2,\dots,n\}$ so that its non-singleton elements are
    \[
        I\subsetneq S\;\text{such that}\; a_{d+1}+\sum_{i\in I}a_i> 1.
    \]
    Then, $\overline{P}_{d,n}^\mathcal{A}$ is a toric variety, and its corresponding fan is that one of the $d$\textsuperscript{th} symmetric product of $\mathcal{B}_\mathcal{A}$.
\end{theorem}

\begin{proof}
    Observe that the building set $\mathcal{B}_\mathcal{A}$ is defined so that $I\in\mathcal{B}_\mathcal{A}$ if and only if $H_I\in\mathcal{G}_\mathcal{A}$. The rest of the argument is identical to that in the proof of Theorem~\ref{main thm1}.
\end{proof}

\begin{remark}
    The condition $a_{d+2} + \cdots + a_n \leq 1$ ensures that $\mathcal{G}_\mathcal{A} \subseteq \mathcal{G}_{\mathcal{A}_{LM}}$, which guarantees the resulting variety $\overline{P}_{d,n}^\mathcal{A}$ is toric and it admits a reduction map from $\overline{P}_{d,n}^\mathcal{LM}$.
\end{remark}

\subsection{Proof of Theorem~\ref{main thm3}}

We now establish the relationship between the GR compactifications $\overline{P}_{d,n}^\mathcal{A}$ and Hassett's compactifications of $M_{0,n}$, proving Theorem~\ref{main thm3}.

\begin{theorem}[{Theorem~\ref{intro: main thm3}}]\label{main thm3}
    Let $\mathcal{A}=(a_1,\dots,a_n)$ be a weight vector such that $a_{d+2}+\cdots+a_n\leq 1$, and let $\mathcal{B}_\mathcal{A}$ be the building set from Theorem~\ref{main thm2}. Then $\mathcal{A}' = (1,a_{d+1},\dots,a_n)$ defines a Hassett compactification $\overline{M}_{0,\mathcal{A}'}$ satisfying:
    \begin{enumerate}
        \item $\overline{M}_{0,\mathcal{A}'}$ admits a reduction map to the compactification of $M_{0,n-d+1}$ given by $\mathbb{P}^{n-d-2}$, and a reduction map from $\overline{M}_{0,n-d+1}^{LM}$.
        \item $\overline{M}_{0,\mathcal{A}'}$ is the toric variety constructed from the nested fan of $\mathcal{B}_\mathcal{A}$.
    \end{enumerate}
\end{theorem}

\begin{proof}
First we prove that the weight vector $\mathcal{A}'=(1,a_{d+1},a_{d+2},\dots,a_n)$ is indeed a valid weight vector in $\mathcal{D}_{0,n}$, so it defines a Hassett space. Note that
\begin{align*}
    1+\sum_{i=d+1}^n a_i \geq 1 + \sum_{i=d+1}^n w_i    & = 1 + (1 - (n-(d+1))\varepsilon + d\varepsilon') + \sum_{i=d+2}^n\varepsilon \\
                                                        & = 2 +d\varepsilon' > 2,
\end{align*}
so that $\mathcal{A}'$ is indeed an element of $\mathcal{D}_{0,n}$. 

Let us next verify the existence of the reduction maps. Note that if $d=1$ and $n=3$ we have that $\overline{P}_{1,3}^\mathcal{A}=\mathbb{P}^1$ for all weight vectors (even if the inequality is not satisfied), so the result is immediate. In the remaining cases, $d>1$ or $d=1$ and $n\geq 4$, it is possible to verify that
\[
    \mathcal{A}_\mathbb{P} = (1,1/(n-2),\dots,1/(n-2)) \leq (1,w_{d+1},\dots,w_n) \leq \mathcal{A}'.
\]
On the other hand, the inequality guarantees the weight vector $(1,1,a_{d+3},\dots,a_n)$ lies in the same coarse chamber as $\mathcal{A}_{LM}= (1,1,1/(n-2),\dots,1/(n-2))$, and this weight vector is greater than or equal to $\mathcal{A}'$. The result follows.

The second part of the theorem is a direct consequence of \cite[Theorem~1.1]{bown2024hypergraphassociahedracompactificationsmoduli} and the definition of $\mathcal{B}_\mathcal{A}$.
\end{proof}

\bibliographystyle{plain}
\bibliography{./bibliography}
\end{document}